\documentclass[twoside,10pt]{article}
\usepackage[ansinew]{inputenc}
\usepackage[T1]{fontenc}
\usepackage{amsmath}
\usepackage{amsthm}
\usepackage{amscd, amsfonts, amssymb}
\usepackage[all]{xy}
\usepackage{mathrsfs}
\usepackage[dvips]{graphicx}
\usepackage{epic, eepic}
\usepackage{titling}
\usepackage{titlesec}
\usepackage[dotinlabels]{titletoc}
\usepackage{abstract}
\usepackage{mathptmx}       
\usepackage{helvet}         
\usepackage{courier}        
\usepackage{fix-cm}

\DeclareSymbolFont{cmletters}{OML}{cmm}{m}{it}
\DeclareSymbolFont{cmsymbols}{OMS}{cmsy}{m}{n}
\DeclareSymbolFont{cmlargesymbols}{OMX}{cmex}{m}{n}
\DeclareMathSymbol{\myjmath}{\mathord}{cmletters}{"7C}
\DeclareMathSymbol{\myamalg}{\mathbin}{cmsymbols}{"71}
\DeclareMathSymbol{\mycoprod}{\mathop}{cmlargesymbols}{"60}
\let\jmath\myjmath

\numberwithin{equation}{section}

\theoremstyle{definition}
\newtheorem{dfn}{Definition}[section]
\newtheorem{example}{Example}[section]

\newtheorem{remark}{Remark}[section]
\theoremstyle{definition}
\newtheorem{prop}{Proposition}[section]
\newtheorem{thm}[prop]{Theorem}

\newtheorem{corollary}[prop]{Corollary}
\newtheorem{lem}[prop]{Lemma}
\newtheorem*{thm*}{Theorem}

\newcommand{\Char}{\mathrm{Ch}}
\newcommand{\Irr}{\mathrm{Irr}}

\newcommand{\AL}{\mathbf{L}}
\newcommand{\Gam}{\mathnormal{\Gamma}}
\newcommand{\ep}{\varepsilon}

\newcommand{\rrangle}{\rangle\!\!\rangle}
\newcommand{\llangle}{\langle\!\!\langle}

\newcommand{\Com}{\mathsf{Com}}
\newcommand{\Mod}{\mathsf{Mod}}

\newcommand{\Der}{\mathrm{Der}}
\newcommand{\SDer}{\mathscr{D} e r}

\DeclareMathOperator{\id}{id}

\DeclareMathOperator{\ob}{ob}

\DeclareMathOperator{\Ann}{Ann}
\DeclareMathOperator{\Aut}{Aut}
\DeclareMathOperator{\End}{End}
\DeclareMathOperator{\Hom}{Hom}
\DeclareMathOperator{\im}{im}
\DeclareMathOperator{\Gal}{Gal}

\DeclareMathOperator*{\dirlim}{\underset{\longrightarrow}{\lim}}

\begin{document}
\pretitle{\begin{flushleft}\LARGE\sffamily} 
\posttitle{\par\end{flushleft}\rule[8mm]{\textwidth}{0.1mm}}
\preauthor{\begin{flushleft}\large\scshape\vspace{-8mm}}
\postauthor{\end{flushleft}\vspace{-8mm}}
\title{Arithmetic hom-Lie algebras, $L$-functions and $p$-adic Hodge theory} 
\author{Daniel Larsson\\
\vspace{0.2cm}
\small Buskerud and Vestfold University College\\
  Pb. 235, 3603
 Kongsberg, Norway\\
 \textnormal{\texttt{daniel.larsson@hbv.no}}
 \normalsize\normalfont}
\date{}


\maketitle
\begin{onecolabstract}
\noindent We construct examples of hom-Lie algebras in some arithmetic contexts. 
\end{onecolabstract}
\tableofcontents
\subsection*{Notations.} The following notations will be adhered to
throughout. 
\begin{itemize}
 \item[-] $k$ will denote a commutative, associative integral domain with unity. 
 \item[-] $\Com(k)$, $\Com(B)$ e.t.c, the category of, commutative,
 associative $k$-algberas ($B$-algebras, etc) with unity. Morphisms of
 $k$-algebras ($B$-algebras, e.t.c) are always unital, i.e.,
 $\phi(1)=1$. 
 \item[-] $A^\times$ is the set of units in $A$ (i.e., the set of
 invertible elements). 
 \item[-] $\Mod(A)$, the category of $A$-modules. 
 \item[-] $\End(A):=\End_k(A)$, the $k$-module of algebra morphisms on $A$. 
 \item[-] $\circlearrowleft_{a,b,c}(\,\cdot\,)$ will mean
cyclic addition of the expression in bracket.
\item[-] When writing actions of group elements we will use the notations
  $\sigma(a)$ and $a^\sigma$, meaning the same thing: the action of $\sigma$ on $a$. 
\end{itemize}
\section{Introduction}The humble aim of this paper is to try to lure some number theorists to believing that a structure known as hom-Lie algebras might be interesting to study in number theory and arithmetic geometry. 

As a short historical account of the emergence of hom-Lie algebras, you could say that it started, like many other things, by Euler. I don't know who exactly introduced the main basic protagonist for this note, namely the notion of $q$-difference (or differential) operator, but I know that F.H Jackson studied these operators and their inverses, ``$q$-integrals'', quite extensively in the early twentieth century. The reason for his interest was that $q$-difference operators approximate ordinary derivations to arbitrary high precision, and provides a discrete analogue of derivations. Therefore these operators have found important applications in numerical analysis, for instance. 

In addition, $q$-analogues of different objects turn up on a day-to-day basis in combinatorics, special functions, functions over finite fields, quantum groups and mathematical physics (``$q$-deformations'') to name a few areas. And of course in number theory. However, it is only relatively recently that $q$-difference operators have begun to interest arithmetically inclined researchers. In the nineties, Yves Andr\'e, Lucia di Vizio and other people started investigating $q$-calculus in the context of number theory and arithmetic geometry. 

However, we were unaware of these possible number-theoretic applications at the time we introduced hom-Lie algebras \cite{HaLaSi}. We were instead motivated by $q$-deformation algebras appearing in physics (mostly from quantum field theory). One such example is the Virasoro algebra, which in the centreless version (also known as the Witt algebra) make an appearance later in this note. It was known that many of these $q$-deformations were actually coming from algebras of $q$-deformed derivations, i.e., $q$-difference operators, but a general underlying algebraic structure was lacking. 

Observing (we were of course not the first ones to do this) that $q$-difference operators are special cases of so-called $\sigma$-twisted derivations ($\sigma$-derivations, $\sigma$-differential operators, \dots), with $\sigma$ an endomorphism on some algebra underlying the structure (algebras of functions on some space), we defined a twisted multiplication of such $\sigma$-derivations. From the resulting structure abstracted the notion of hom-Lie algebra (and later for more general $\sigma$-derivations, quasi-hom-Lie algebras and more generally quasi-Lie algebras). Later other people generalized this in different directions, but the notion originates as algebras of $\sigma$-twisted derivations.

So, it is not especially surprising that this evaded the interest of number theorist,s as the motivation for introducing hom-Lie algebras was very far from number theory. In addition, we ourselves never had any thoughts in that direction since we weren't trained as number theorists but as algebraists/geometers. However, when I taught a course in number theory in 2008, I came to realize that hom-Lie algebras are really very natural objects to study in arithmetic contexts. This note is therefore mostly propaganda.

The contents of the paper is well-described by the table of contents so I won't go into any details here. 

This paper is a ``companion paper'' to the paper \cite{LarArithom} where more details concerning hom-Lie algebras are given. Here we only very briefly indicate the main points and refer to \cite{LarArithom} for full details.

\subsubsection*{Acknowledgement} This paper is an expanded version of a letter to J\"urgen Ritter trying to explain the main ideas of hom-Lie algebras together with examples from the theory of $L$-functions and Iwasawa theory. I wish to thank him for useful comments and ideas.  
\section{Hom-Lie algebras and twisted derivations}
\subsection{Generalities}
Let $A\in\ob(\Com(k))$ and let $\sigma: A\to A$ be a $k$-linear map on
$A$. Then a \emph{twisted derivation} on $A$ is a $k$-linear map
$\partial: A\to 
A$ satisfying $$\partial(ab)=\partial(a)b+\sigma(a)\partial(b).$$
We can generalize this as follows. Let $A$ and $\sigma$ be as above, and
$M\in\ob(\Mod(A))$. The action of $a\in A$ on $m\in M$ will be denoted
$a.m$. Then, a \emph{twisted derivation on $M$} is $k$-linear map
$\partial: M\to M$ such that 
\begin{equation}\label{eq:Leibniz}
\partial(a.m)=\partial_A(a).m+\sigma(a).\partial(m),
\end{equation} where, by
necessity, $\partial_A:A\to A$ is a twisted derivation on $A$ (in the
first sense). We call $\partial_A$ the \emph{restriction of $\partial$
  to $A$}. Finally, a \emph{twisted module derivation} is a
$k$-linear map $\partial : A\to M$ such that 
$$\partial(ab)=b.\partial(a)+\sigma(a).\partial(b),$$ for
$\sigma\in\End(A)$.  Normally we will not differentiate between left and right modules structures, but there are times when such a distinction would be necessary. 

We will sometimes refer to the above as \emph{$\sigma$-twisted
  (module) derivations} if we want to emphasize which $\sigma$ we 
  refer to. 

Let $\sigma\in\End(A)$ and denote by $A^{(\sigma)}:=A\otimes_{A,\sigma} A$, the extension of scalars along $\sigma$. This means that we consider $A$ as a left module over itself via $\sigma$, i.e., $a.b:=\sigma(a)b$. The right module structure is left unchanged. If $M$ is an $A$-module, we put 
$$M^{(\sigma)}:=A^{(\sigma)}\otimes_A M=A\otimes_{A,\sigma} M,$$i.e., $M$ is endowed with left module structure $a.m:=\sigma(a)m$, and once more, the right structure is unaffected.

We note that a $\sigma$-derivation $d_\sigma$ on $A$ is actually a derivation $d: A\to A^{(\sigma)}$ and conversely. Indeed,
$$d(ab)=d(a)b+a.d(b)=d(a)b+\sigma(a)d(b).$$In the same manner, a $\sigma$-derivation $d_\sigma: A\to M$ is a derivation $d: A\to M^{(\sigma)}$, and conversely. 

Therefore, there is a one-to-one correspondence between $\sigma$-derivations $d_\sigma:A\to M$ and derivations $d:A\to M^{(\sigma)}$. 
\begin{example}The ``universal'' (this designation will be amply demonstrated in what follows) example of a $\sigma$-derivation is the following. Let $A\in\ob(\Com(k))$ and
  $M\in\ob(\Mod(A))$. Suppose 
  $\sigma: M\to M$ is $\varsigma$-semilinear, i.e.,
  $\sigma(a.m)=\varsigma(a).\sigma(m)$, for $a\in A$ and $m\in M$, where
  $\varsigma\in\End(A)$. Then, a small computation shows that 
  for all $b\in A$, 
  $\partial:=b(\id-\sigma): M\to M$, is a $\varsigma$-twisted derivation on
  $M$.  
Notice that if $M=A$, we automatically get
  $\varsigma=\sigma$. On the other hand, given a $\sigma$-twisted derivation
  $\partial=(\id-\sigma): M\to M$, with $\sigma$ $\varsigma$-semilinear, there
  is a $\varsigma$-linear map defined by 
  $\sigma=\id-\partial$. More generally, for $a\in A^\times$,
  $\partial_a=a(\id-\sigma)$, defines a $\varsigma$-semilinear map
  $\sigma=\id-a^{-1}\partial_a$. Hence, there is a duality (in
  some sense) between twisted derivations and semilinear
  maps.

\end{example}

\subsection{Algebras of twisted derivations}\label{sec:moduletwisted}
Let $M$ be an $A$-module. Then the $k$-modules of $\sigma$-twisted derivations,
\begin{align*}
\mathrm{Der}_{\sigma}(M)&:=\{\partial\in\End_k(M)\mid
  \partial(a.m)=\partial_A(a).m+\sigma(a).\partial(m)\},\quad\text{and}\\
  \mathrm{Der}_\sigma(A,M)&:=\{\partial\in\Hom_k(A,M)\mid
  \partial(ab)=\partial(a).b+\sigma(a).\partial(b)\}
\end{align*} are left $A$-modules
 
Note that unlike the case of ordinary derivations, $\mathrm{Der}_\sigma(M)$
or $\mathrm{Der}_\sigma(A,M)$ are not Lie algebras. 
\subsubsection{Equivariant hom-Lie algebras}
Let $G$ be a group, $A$ a $k$-algebra and $M$ an $A$-module. 
\begin{dfn}\label{dfn:globhom}An \emph{equivariant hom-Lie
  algebra} for $G$ over
  $A$ is a $(G-A)$-module $M$ together
  with a $k$-bilinear product $\llangle\,\cdot,\cdot\,\rrangle$ on
  $M$ such that 
  \begin{itemize} 
       \item[(hL1.)] $\llangle a,a\rrangle=0$, for all $a\in M$;
       \item[(hL2.)] $\circlearrowleft_{a,b,c}\big (\llangle a^\sigma,\llangle
       b,c\rrangle\rrangle+q_\sigma\cdot\llangle a,\llangle
       b,c\rrangle\rrangle\big)=0$, for all $\sigma\in G$ and some $q_\sigma\in A$. 
  \end{itemize} A morphism of equivariant hom-Lie algebras $(M, G)$ and
       $(M', G')$ is a pair $(f,\psi)$ of a morphism of
       $k$-modules $f: M\to M'$ and $\psi: G\to G'$ such 
       that $f\circ \sigma=\psi(\sigma)\circ f$, and $f\llangle
       a,b\rrangle_{M} =\llangle  
       f(a),f(b)\rrangle_{M'}$. 
\end{dfn}Notice that the definition implies that for a morphism $$(f,\psi): (M,G)\to (M', G')$$ we must have $f(q_\sigma)=q_{\psi(\sigma)}$. 

We get a \emph{hom-Lie algebra} when we simply consider one $\sigma\in G$. Notice that for $\sigma=\id$ we get a Lie algebra. 

By an \emph{arithmetic (equivariant) hom-Lie algebra} I mean, loosely speaking, a hom-Lie structure over some arithmetically interesting ring $A$. 

\begin{remark}
Actually, one can make sense of the proposition that hom-Lie algebras are Lie algebras ``in a suitably twisted category''. 
\end{remark}
\subsubsection{Products}
Let, as before, $A\in\ob(\Com(k))$ and let $\sigma\in\End(A)$. Denote
by $\delta_\sigma$ a $\sigma$-twisted derivation on $M$ whose
restriction to $A$ is $\partial$, i.e.,
$\delta_\sigma\in\mathrm{Der}_\sigma(M)$ and
$\partial\in\mathrm{Der}_\sigma(A)$. Assume that 
$\sigma(\Ann(\delta_\sigma))\subseteq \Ann(\delta_\sigma)$, where
$$\Ann(\delta_\sigma):=\{a\in A\mid a\delta_\sigma(m)=0,\quad \text{for
  all}\quad m\in M\},$$ and that
 \begin{equation}\label{eq:qsigmapartial}
 \partial\circ\sigma=q\cdot\sigma\circ\partial,\qquad\text{for some $q\in 
  A$.}
  \end{equation} Form the left $A$-module 
$$A\cdot\delta_\sigma:=\{a\cdot\delta_\sigma\mid a\in A\}.$$ Define
\begin{equation}\label{eq:prodtwist}
\llangle a\cdot\delta_\sigma,b\cdot\delta_\sigma\rrangle:=
\sigma(a)\cdot\delta_\sigma(b\cdot\delta_\sigma)-
\sigma(b)\cdot\delta_\sigma(a\cdot\delta_\sigma). 
\end{equation}This should be interpreted as 
$$
\llangle a\cdot\delta_\sigma,b\cdot\delta_\sigma\rrangle(m):=
\sigma(a)\cdot\delta_\sigma(b\cdot\delta_\sigma(m))-
\sigma(b)\cdot\delta_\sigma(a\cdot\delta_\sigma(m)),$$for $m\in M$. 
We now have the following fundamental theorem.
\begin{thm}\label{thm:twistprod}Under the above assumptions,
  equation (\ref{eq:prodtwist}) gives a 
  well-defined $k$-linear product on $A\cdot\delta_\sigma$ such that 
\begin{itemize}
\item[(i)] $\llangle a\cdot\delta_\sigma, b\cdot\delta_\sigma\rrangle=
  (\sigma(a)\partial(b)-\sigma(b)\partial(a))
  \cdot\delta_\sigma$;  
\item[(ii)] $\llangle a\cdot \delta_\sigma, a\cdot\delta_\sigma\rrangle=0$;
\item[(iii)] $\circlearrowleft_{a,b,c}\big(\llangle \sigma(a)\cdot
  \delta_\sigma,\llangle
  b\cdot\delta_\sigma,c\cdot\delta_\sigma\rrangle\rrangle+ q\cdot\llangle
  a\cdot\delta_\sigma,\llangle
  b\cdot\delta_\sigma,c\cdot\delta_\sigma\rrangle\rrangle\big)=0$,
\end{itemize} where, in (iii), $q$ is the same as in (\ref{eq:qsigmapartial}).
\end{thm}

\begin{corollary}In the case $\delta_\sigma\in\mathrm{Der}_\sigma(A,M)$, defining the algebra structure directly by property (i) in the theorem gives (ii) and (iii) on $A\cdot\delta_\sigma$. 
\end{corollary}

We can extend $\sigma$ to an algebra morphism on $A\cdot\delta_\sigma$ by defining $\sigma(a\cdot\delta_\sigma):=\sigma(a)\cdot\delta_\sigma$.

In case $M=A$ is a unique factorization domain (UFD), it is possible (see \cite{HaLaSi}) to
prove the following:
\begin{thm}\label{thm:UFD} If $A$ is a UFD, and $\sigma\in\End(A)$, then
$$\delta_\sigma:=\frac{\id-\sigma}{g}$$
generates $\mathrm{Der}_\sigma(A)$ as a left $A$-module, where
$g:=\gcd((\id-\sigma)(A))$. 
\end{thm}

\begin{example}When $A=K$ is a field the above theorem
  implies that every $\sigma$-twisted derivation is on the form given
  in the statement. In fact, every  $\sigma$-derivation is on the form $a(\id-\sigma)$. 
\end{example}

\section{Equivariant hom-Lie algebras in three arithmetic contexts}

There are three constructions that I find particularly interesting in this context and they are very closely related: $L$-functions, Galois modules and Iwasawa theory. Let me take them in turn. In fact, they all centres around $L$-functions in some form ($p$-adic or ordinary). 
\subsection{$L$-functions}

Let $K/k$ be a finite Galois extension (that the extension is finite is not necessary, but let's assume this for simplicity), with Galois group $G$. In addition let $E$ be a realization field for $G$, i.e., a field over which all irreducible representations of $G$ are defined. We can assume that $E/\mathbb{Q}$ is Galois with group $\Gam$. It is reasonable that we chose $E$ as small as possible. 

The construction below also work in the case of $S$-truncated $L$-functions but for simplicity of notation we assume that we have the full $L$-functions. 

There is an action of $\Gam$ on every character $\chi$ of $G$ defined by 
$$\chi^\sigma:=\sigma\circ \chi, \qquad \sigma\in\Gam.$$ If $\chi$ is irreducible, then so is $\chi^\sigma$ (at least when $G$ is finite). Therefore, $\Gam$ acts on the set of irreducible characters of $G$, and the construction below highlights the interplay between $G$ and $\Gam$. 

Let $R$ be a subring of $\mathbb{C}$ and let $\AL_R$ be the algebra generated as 
$$\AL_R:=\mathrm{Sym}^\#\big(\bigoplus_{\chi\in\Irr(G)} R\cdot L_\chi\big)=\mathrm{Sym}^\#\big(\bigoplus_{\xi\in \Char(G)}R\cdot L_\xi\big),$$ where the symbol $\#$ indicates that we impose the relations 
$$L_{\chi+\vartheta}=L_\chi\cdot L_\vartheta$$ in the algebra structure of the symmetric algebra. By $\Char(G)$ we mean the group of (virtual) characters of $G$:
$$\Char(G)=\bigoplus^{|\Irr(G)|}_{i=1} \mathbb{Z} \chi_i, \quad \text{where $\chi_i\in\Irr(G)$}.$$ At this point the $L_\chi$ are only formal symbols.

We have a natural sequence of maps
$$\xymatrix{\bigoplus_{\chi\in\Irr(G)} R\cdot L_\chi\ar[r]&  \bigoplus_{\xi\in \Char(G)} \mathbb{C}\cdot L_\xi\ar[r]^{sp}\ar@{..>}[d]&
\mathrm{Frac}(\mathbb{C}[\![s]\!])\\
&\Hom(\Char(G),\mathrm{Frac}(\mathbb{C}[\![s]\!])^\times),}$$ where the first map is the natural inclusion, the vertical is $$L_\xi\longmapsto \begin{cases}
\theta\mapsto L(\theta,s) & \text{if } \theta=\xi\\
\theta\mapsto 1 & \text{otherwise},\end{cases}$$ and last map is the specialization\footnote{This assumes that Artin's conjecture on the analyticity of $L$-functions is true. However, in the case of negative integral points $s$ this is ok since the $L$-functions vanish at these points.} map $$sp: \,\, \AL_\mathbb{C}\longrightarrow
\mathrm{Frac}(\mathbb{C}[\![s]\!]),\quad L_\xi\mapsto L(\xi,s).$$ I want to stress that this map is dependent on a choice of variable $s$, but let us ignore this. 

This sequence of maps induce the following sequence
$$\xymatrix{\mathrm{Sym}\big(\bigoplus_{\chi\in\Irr(G)} R\cdot L_\chi\big)\ar[r]]&  \mathrm{Sym}\big(\bigoplus_{\xi\in \Char(G)} \mathbb{C}\cdot L_\xi\big)\ar[r]^<<<<<<{sp}\ar@{..>}[d]&
\mathrm{Frac}(\mathbb{C}[\![s]\!])\\
&\Hom(\Char(G),\mathrm{Frac}(\mathbb{C}[\![s]\!])^\times).}$$ The horizontal arrows are ring morphisms but the dotted arrow are not since the group $\Hom(\Char(G),\mathrm{Frac}(\mathbb{C}[\![s]\!])^\times)$ is not a ring. This induces $$\xymatrix{\AL_R=\mathrm{Sym}^\#\big(\bigoplus_{\chi\in\Irr(G)} R\cdot L_\chi\big)\ar[r]^{\hookrightarrow}&  \mathrm{Sym}^\#\big(\bigoplus_{\xi\in \Char(G)} \mathbb{C}\cdot L_\xi\big)\ar[r]^<<<<<<{sp}\ar@{..>}[d]&
\mathrm{Frac}(\mathbb{C}[\![s]\!])\\
&\Hom(\Char(G),\mathrm{Frac}(\mathbb{C}[\![s]\!])^\times),}$$and putting $R=\mathbb{C}$ we get
$$\xymatrix{\AL_\mathbb{C}=\mathrm{Sym}^\#\big(\bigoplus_{\chi\in\Irr(G)} R\cdot L_\chi\big)\ar@{=}[r]\ar@{..>}[dr]&  \mathrm{Sym}^\#\big(\bigoplus_{\xi\in \Char(G)} \mathbb{C}\cdot L_\xi\big)\ar[r]^<<<<<<{sp}\ar@{..>}[d]&
\mathrm{Frac}(\mathbb{C}[\![s]\!])\\
&\Hom(\Char(G),\mathrm{Frac}(\mathbb{C}[\![s]\!])^\times).}$$We now have the set-up ready and we can form the associated (equivariant) hom-Lie algebra. 

Choose a $\sigma\in \Gam$. Extending linearly, this acts on the ring $\AL_R$ by postulating that $L_\chi^\sigma=L_{\chi^\sigma}$. In this way $\sigma$ becomes a ring automorphism on $\AL_R$. Now we can form the module
$$\AL_R\cdot \delta_\sigma=\AL_R\cdot a(\id-\sigma),\quad \text{for some $a\in R$}.$$ The natural choice for $a$ is obviously $1$, but in some cases it is necessary or desirable to have other $a$'s. 

Then by the previous theorem this can be endowed with a hom-Lie algebra structure: \begin{align}\label{eq:LhomLie}
\begin{split}
\llangle pL_\chi\cdot \delta_\sigma,qL_\vartheta\cdot \delta_\sigma\rrangle:&=\big(p^\sigma L_\chi^\sigma L_\vartheta-q^\sigma L_\vartheta^\sigma L_\chi\big)\cdot\delta_\sigma\\
&=\big(p^\sigma L_{\sigma\circ\chi+\vartheta}-q^\sigma L_{\sigma\circ \vartheta+\chi}\big)\cdot\delta_\sigma, \quad p,q\in R.
\end{split}
\end{align} Notice that we can let $\sigma$ act on $R$ if this is what we want ($R$ could be the realization field $E$ for instance). Letting $\sigma$ vary over $\Gam$ we get the associated $\Gam$-equivariant hom-Lie algebra. 

In addition, we note that the defined (equivariant) structure is only formal as we are simply playing with symbols having the same properties (actually only one such property, which is all we need) as Artin $L$-functions. We want to treat them like \emph{functions}, and then it is reasonable to use the map  $sp: \AL_R\to \mathrm{Frac}(\mathbb{C}[\![s]\!])$. Applying $sp$ to (\ref{eq:LhomLie}) yields, with $p,q\in R$ again,
\begin{align}\label{eq:LhomLiefunc}
\begin{split}
\llangle pL(\chi,s)\cdot \delta_\sigma,qL(\vartheta,s)\cdot \delta_\sigma\rrangle:&=\big(p^\sigma L(\chi,s)^\sigma L(\vartheta,s)-q^\sigma L(\vartheta,s)^\sigma L(\chi,s)\big)\cdot\delta_\sigma\\
&=\big(p^\sigma L(\sigma\circ\chi+\vartheta,s)-q^\sigma L(\sigma\circ \vartheta+\chi,s)\big)\cdot\delta_\sigma.
\end{split}
\end{align}This is still not quite correct as the $L(\chi,s)$ as functions are independent on the choice of Taylor expansion (i.e., the choice of base point for the expansion), but let's ignore this point also in this exposition. This can clearly be made into correct statements. 

Let us denote the hom-Lie algebra given by (\ref{eq:LhomLiefunc}) as $\underline{h}\AL_{R}^{(s)}(\sigma)$, and the associated equivariant structure $\underline{h}\AL_{R}^{(s)}(\Gam)$. Notice that this is actually a functor from Galois realization fields for $K/k$ to equivariant hom-Lie algebras (which also can be made precise; at least I think so).

Ok, this is certainly interesting in itself (or at least \emph{I} think so), but it becomes even more interesting if we pass to special values. We denote the leading coefficient in the Taylor expansion of $L(\chi,s)$ by $L^\ast(\chi,s)$ as is customary.

The operation $(?)^\ast :\mathrm{Frac}(\mathbb{C}[[s]])\to \mathbb{C}$ is a ring morphism and extends uniquely to a morphism of the associated hom-Lie algebra $\underline{h}\AL_R^{(s)}(\sigma)$ (we fix a $\sigma$ for simplicity). Therefore, there is an induced hom-Lie algebra structure on the set of leading coefficients of all Artin $L$-functions attached to $K/k$. We denote this by $h{\AL}^{(s),\ast}_R(\sigma)$. 

In addition, since Artin root numbers $W(\chi)$ satisfies the same relation $W(\chi+\vartheta)=W(\chi)W(\vartheta)$, the same applies to these and we can form $\underline{h}\mathbf{W}^{(s)}_R(\sigma)$. 

I find this construction rather remarkable since it incorporates all special values of the $L$-functions into one algebraic structure which is very similar to Lie algebras. I would say that this merits some further study (ideally be someone more knowledgeable about special values than me; such a person shouldn't be too hard to find, but getting him/her interested in this construction might prove more difficult). 

\begin{example}Let $D$ be a group of Dirichlet characters $\chi: (\mathbb{Z}/m)^\times\to \mathbb{C}$ and let $K_D$ be the associated field. The degree of $K_D$ is $n:=|D|$ and $\Gal(K_D/\mathbb{Q})=D$. Put $E:=\im(D)$ and assume that $E/\mathbb{Q}$ is Galois with group $\Gam$. 

The associated Dirichlet $L$-functions $L_D(\chi,s)$ are Artin $L$-functions of the field $K_D/\mathbb{Q}$. Therefore, $L_D(\chi+\vartheta,s)=L_D(\chi,s)L_D(\vartheta,s)$, for $\chi,\vartheta\in D$, and hence the above construction applies to $\underline{h}\AL_{D,R}^{(s)}(\sigma)$, $\sigma\in\Gam$, the space of Dirichlet $L$-functions associated to $D$, giving a product as in (\ref{eq:LhomLiefunc}). 

If $s=1-n$, $n\geq 1$, we can give a direct and explicit formula involving generalized Bernoulli numbers. In fact, directly from the definition and by \cite[Theorem 4.2]{Washington} we get
\begin{multline*}\big(\llangle L_D(\chi,1-n)\cdot\delta_\sigma,L_D(\vartheta,1-n)\cdot\delta_\sigma\rrangle\big)^\ast\\
=\big(L_D(\sigma\circ\chi,1-n)L(\vartheta,1-n)-L_D(\sigma\circ\vartheta,1-n)L_D(\chi,1-n)\big)^\ast\\
=-\frac{B_{n,\sigma\circ\chi}}{n}\frac{B_{n,\vartheta}}{n}+\frac{B_{n,\sigma\circ\vartheta}}{n}\frac{B_{n,\chi}}{n}
=\frac{B_{\sigma\circ\vartheta}B_{n,\chi}-B_{n,\sigma\circ\chi}B_{n,\chi}}{n^2},
\end{multline*}where $B_{n,\chi}$ are the generalized Bernoulli numbers defined in \cite[p. 31]{Washington}.
\end{example}

\subsection{Galois modules}This is much simpler than the previous example. 

Let $M$ be a finitely generated module over the $k$-algebra $A$. We can always construct an $A$-algebra structure on $M$, which we denote $M^a$ for emphasis. For instance, we could always take $\mathrm{Sym}(M)=\mathrm{Sym}_A(M)$. Assume such an algebra structure is given and let the group $G$ (no finiteness-assumption) act on $A$ and $M$ linearly over $k$ such that the action on $A$ and $M$ is compatible in the usual sense:
$$\sigma(am)=\sigma(a)\sigma(m), \quad a\in A, \,\,\, m\in M.$$ The action of $G$ on $A$ may be trivial. 

We can now form the equivariant hom-Lie algebra structure
$$\underline{h} M(G):=\bigsqcup_{\sigma\in G}\underline{h}M(\sigma):=\bigsqcup_{\sigma\in G} M^a\cdot \delta_{a_\sigma}, \quad \delta_{a_\sigma}:=a_\sigma(\id-\sigma), \,\, a_\sigma\in A.$$ The product is of course once again given by Theorem \ref{thm:twistprod}. Clearly, $\underline{h}M(G)$ includes all the structure that the original $G$-representation does, at least when $M^a=\mathrm{Sym}(M)$. In fact, it includes more: it encodes the \emph{relative} action of $\sigma\in G$ on the elements of $M$. When the algebra structure $M^a$ is different, other hom-Lie algebra structure emerge, arranging the information of the original $G$-representation in more sophisticated ways. \emph{However}, let me point out that it might be very difficult to construct $M^a$ since it has to be compatible with the original action of $G$. 

\begin{example}
 Take the module $M$ to be 
 $$M=\bigoplus_{i=1}^n \mathbb{Z}\ep_i,$$ the free $\mathbb{Z}$-module of rank $n$. Then we can form the localization $$M^a_{(\underline{\ep})}=\mathrm{Sym}(M)_{(\underline{\ep})}=A[\mathbb{Z}^{\oplus n}]=A[\ep_1^{\pm 1},\dots,\ep_n^{\pm 1}],$$ the $A$-algebra over Laurent polynomials over $A$. 
 
There is a more or less complete classification of the hom-Lie algebra structures appearing in this context by computations in \cite{HaLaSi}.
\end{example}This construction can be used as in the following example. 
\begin{example}
Take $K$ to be a field with three fundamental units $\ep_1$, $\ep_2$, $\ep_3$, let $\mu(K)$ denote the group of roots of unity in $K$. Then put $$A=\mathbb{Z}[\mu(K)][\ep_1^{\pm 1},\ep_2^{\pm 1},\ep_3^{\pm 1}].$$ We introduce the notation $\underline{x}:=(x_1,x_2,x_3)$ for any vector. 

Put $\sigma(\ep_i):=\zeta_i\ep_1^{s_{i1}}\ep_2^{s_{i2}}\ep_3^{s_{i3}}$, with $\zeta_i\in\mu(K)$.  So $\sigma$ can be represented by the matrix
$$\Sigma:=(s_{ij})=\begin{pmatrix}
	s_{11} & s_{12} & s_{13}\\
	s_{21} & s_{22} & s_{23}\\
	s_{31} & s_{32}	& s_{33}
\end{pmatrix}$$ together with the triple $\underline{\zeta}:=(\zeta_1,\zeta_2,\zeta_3)$. We consider $\sigma$ acting on $A$ over $\mathbb{Z}[\mu(K)]$ (to simplify things; there is no need to do so in principle) so in fact we look at $$A=\mathbb{Z}[\mu(K)]\otimes_\mathbb{Z}\mathbb{Z}[\ep_1^{\pm 1},\dots, \ep_n^{\pm 1}],$$ with $\sigma=1\otimes\sigma$. 

A common divisor on $A$ is given by $$g:=\omega^{-1}\underline{\ep}^{\underline{g}}, \quad \underline{g}:=(g_1,g_2,g_3)$$ and therefore, by Theorem \ref{thm:UFD}, a $\sigma$-twisted derivation is given by 
$$\delta_\sigma:=\omega\underline{\ep}^{-\underline{g}}(\id-\sigma),$$ and this generates a submodule $A\cdot\delta_\sigma\subseteq\mathrm{Der}_\sigma(A)$. The notation $\underline{\ep}^{\underline{g}}$ obviously mean $\ep_1^{g_1}\ep_2^{g_2}\ep_3^{g_3}$, and we use this notation throughout. 

We need some more notation before we can state the result. Put
\begin{align*}
(\underline{y})_r&:=s_{1r}y_1+s_{2r}y_2+s_{3r}y_3\\
 d_1&:=(s_{11}-1)g_1+s_{21}g_2+s_{31}g_3\\
 d_2&:=s_{12}g_1+(s_{22}-1)g_2+s_{32}g_3\\ 
 d_3&:=s_{13}g_1+s_{23}g_2+(s_{33}-1)g_3 \\
{D}_{\underline{\ell}}&:=-\underline{\ep}^{\underline{\ell}}\delta_\sigma.
\end{align*}

For this data to define a hom-Lie algebra we need that $d_1=d_2=d_3=0$ (see \cite[p.351]{HaLaSi}; but beware that the notation is slightly different there) and this can be expressed as
$$(\Sigma-\id)^\mathrm{t}\underline{g}^{\mathrm{t}}=0\quad\Longleftrightarrow\quad\begin{pmatrix}
	s_{11}-1& s_{12} & s_{13}\\
	s_{21} & s_{22}-1 & s_{23}\\
	s_{31} & s_{32}	& s_{33}-1
\end{pmatrix}^{\mathrm{t}}
\begin{pmatrix}
g_1\\
g_2\\
g_3\end{pmatrix}=0.$$With some work one can compute
\begin{align*}
\llangle D_{\underline{k}},D_{\underline{\ell}}\rrangle &= \omega\big(\underline{\zeta}^{\underline{\ell}}D_{(\ell)_1+k_1-g_1,(\ell)_2+k_2-g_2,(\ell)_3+k_3-g_3}
\\
&\qquad -
\underline{\zeta}^{\underline{k}}D_{(k)_1+\ell_1-g_1,(k)_2+\ell_2-g_2,(k)_3+\ell_3-g_3}\big)\\
&=\omega\big(\underline{\zeta}^{\underline{\ell}}\ep_1^{(\ell)_1+k_1-g_1}\ep_2^{(\ell)_2+k_2-g_2}\ep_3^{(\ell)_3+k_3-g_3}\\
&\qquad-\underline{\zeta}^{\underline{k}}\ep_1^{(k)_1+\ell_1-g_1}\ep_2^{(k)_2+\ell_2-g_3}\ep_3^{(k)_3+\ell_3-g_3}\big)\delta_\sigma.
\end{align*}
 
Notice that this algebra structure has infinite rank over $\mathbb{Z}[\mu(K)]$, but obviously rank one over $\mathbb{Z}[\mu(K)][\underline{\ep}^{\pm 1}]$. 

Now, this obviously becomes very complicated in this generality with a generic Galois action (not to mention with more generators), so let us specialize to an explicit example. So let us take 
$$\sigma(\ep_1):=\zeta_1\ep_2,\,\,\, \sigma(\ep_2)=\zeta_2\ep_1^{-1},\,\,\, \sigma(\ep_3)=\zeta_3\ep_3\quad\Longleftrightarrow\quad
\Sigma=\begin{pmatrix}
0 & 1 & 0\\
-1 & 0 & 0\\
0 & 0 & 1
\end{pmatrix}$$In this case it is easy to see that a greatest common divisor is in fact an element in $\mathbb{Z}[\mu(K)]$ so $\delta_\sigma=(\id-\sigma)$ generates the whole $\Der_\sigma(A)$ as an $A$-module. From this follows that we must have $g_1=g_2=g_3=0$, and hence $d_1=d_2=d_3=0$. Also, we have that $(x)_1=-x_2$, $(x)_2=x_1$ and $(x)_3=x_3$. 

Putting all this into the above formula for $\llangle D_{\underline{k}},D_{\underline{\ell}}\rrangle$ we get 
\begin{multline*}
\llangle D_{\underline{k}},D_{\underline{\ell}}\rrangle=\underline{\zeta}^{\underline{\ell}}D_{\ell_2+k_1,\ell_1+k_2,\ell_3+k_3}-\underline{\zeta}^{\underline{k}}D_{k_2+\ell_1,k_1+\ell_2,k_3+\ell_3}\\
=\big(\underline{\zeta}^{\underline{\ell}}\ep_1^{\ell_2+k_1}\ep_2^{\ell_1+k_2}\ep_3^{\ell_2+k_3}-
\underline{\zeta}^{\underline{k}}\ep_1^{k_2+\ell_1}\ep_2^{k_1+\ell_2}\ep_3^{k_3+\ell_3}\big)\delta_\sigma.
\end{multline*}
Once again, notice that the algebra has infinite $\mathbb{Z}[\mu(K)]$-rank. It would therefore be interesting to study which (hom-Lie) subalgebras exist with finite $\mathbb{Z}[\mu(K)]$. 

Clearly, if $W$ is a subalgebra a necessary condition is that $\sigma(W)\subseteq W$. In this case, this can only happen for a module on the form 
$$W={y}\mathbb{Z}[\mu(K)][\ep_3^{\pm 1}]\cdot\delta_\sigma,\quad \text{where $y\in \mathbb{Z}[\mu(K)][\ep_3^{\pm 1}]$}.$$ Hence, $D_\ell=-\ep_3\delta_\sigma$ and we get a ($\mathbb{Z}$-graded!) product
\begin{equation}\label{eq:UnitWitt}\llangle D_k,D_\ell\rrangle =\big(\zeta_3^\ell-\zeta_3^k\big)D_{k+\ell}=
\big(\zeta_3^\ell-\zeta_3^k\big)\ep_3^{k+\ell}\cdot\delta_\sigma.
\end{equation}This is still infinite over $\mathbb{Z}[\mu(K)]$, but there is an interesting finite-rank subalgebra of this, isomorphic to a deformed $\mathfrak{sl}_2$ over $\mathrm{Frac}(\mathbb{Z}[\mu(K)])$. Actually, it is enough to invert $2$: the isomorphism can be defined over 
$\mathbb{Z}[\frac{1}{2}][\mu(K)]$ as we will see. 

Indeed, the $\mathbb{Z}[\mu(K)]$-submodule $$\mathbb{Z}[\mu(K)]D_{-1}\oplus\mathbb{Z}[\mu(K)]D_0\oplus\mathbb{Z}[\mu(K)]D_1$$ is closed under the product (\ref{eq:UnitWitt}):
$$\llangle D_{-1},D_0\rrangle=(1-\zeta_3^{-1})D_{-1},\quad \llangle D_0,D_1\rrangle=(\zeta_3-1)D_1,\quad \llangle D_1,D_{-1}\rrangle=(\zeta_3^{-1}-\zeta_3)D_0.$$

Now, we make the following change of basis
$$D_0\to a^{-1}B_0,\quad D_1 \to b^{-1}B_1,\quad D_{-1}\to c^{-1}B_{-1}.$$ A priori we need to invert $a, b, c$, but we will see that all these can be chosen to be units in $\mathbb{Z}[\mu(K)]$. We compute
\begin{align*}
\llangle B_1,B_0\rrangle &=\llangle bD_1,aD_0\rrangle =ab\llangle D_1,D_0\rrangle = -ab(\zeta-1)D_1\\
\llangle B_{-1},B_0\rrangle &=\llangle cD_{-1},aD_0\rrangle =ac\llangle D_{-1},D_0\rrangle =ac(1-\zeta^{-1})D_{-1}\\
\llangle B_1,B_{-1}\rrangle &=\llangle bD_1,cD_{-1}\rrangle =bc\llangle D_1,D_{-1}\rrangle = bc(\zeta^{-1}-\zeta)D_0.
\end{align*}
What we want to arrive at is 
$$\llangle B_1,B_0\rrangle = 2B_1,\quad \llangle B_{-1},B_0\rrangle = -2qB_{-1},\quad \llangle B_1,B_{-1}\rrangle = \frac{q+1}{2}B_0,$$ for reasons that we be clear in a short while. Here $q$ is an element in $\mathrm{Frac}(\mathbb{Z}[\mu(K)])$ that we want to determine. Therefore, we need to find $a,b,c$ and $q$ such that the transformation exists as we want it to be. 

From the above we find that 
\begin{align*}
	-b(2+a(\zeta-1)=0&\Longrightarrow a=-\frac{2}{\zeta-1}\\
	c(-2q+a(1-\zeta^{-1})=0&\Longrightarrow q=\frac{\zeta^{-1}-1}{\zeta-1}\\
	bc(\zeta^{-1}-\zeta)=\frac{q+1}{2}a&\Longrightarrow bc(\zeta^{-1}-\zeta)=\frac{1}{2}(1-\zeta)\zeta^{-1}\frac{2}{1-\zeta}
	\Longrightarrow bc=\frac{1}{1-\zeta^2}.
\end{align*}(Cyclotomic units!) We can thus choose $b=1$, for instance, giving $c=(1-\zeta^2)^{-1}$, and the isomorphism is established. 

The reason for the change of basis is that the resulting relations are exactly the relations defining a certain $q$-deformation of $\mathfrak{sl}_2$ given in \cite[Example 3.2]{LaSi}. This algebra has some remarkable properties which can be summarized in the following (incomplete) list:
\begin{itemize}
	\item[-] it has a ``universal enveloping algebra'' and a Poincar\'e--Birkhoff--Witt basis;
	\item[-] it is an iterated Ore extension: and as a result has some nice homological properties such as being Auslander-regular and Koszul of Gel'fand--Kirillov dimension 3;
	\item[-] it is isomorphic to a so-called ``down-up algebra'';
	\item[-] homogenizing the defining relations by a central element gives an Artin--Schelter regular algebra and hence a well-established (projective) non-commutative geometry (see the references in \cite{LaSi}).
\end{itemize}
For all this see \cite[Example 3.2]{LaSi}.

\end{example}
Generalizing the above we can state the following theorem:
\begin{thm}Suppose there is a $\sigma$-stable rank-one submodule of $M^a=R[\ep_1^{\pm 1},\dots, \ep_n^{\pm 1}]$, i.e., that $\sigma(R\ep_i)\subseteq R\ep_i$, and that $R^\sigma=R$ then there is a $3$-dimensional hom-Lie subalgebra of $M^a\cdot\delta_\sigma$ which is isomorphic to  a deformed $\mathfrak{sl}_2$ with the above relations and subsequent properties stated above.
\end{thm}The proof is essentially given in the example by direct generalization. 

\begin{corollary}This applies in particular to algebras defined by fundamental units as in the previous example.
\end{corollary}
\subsection{Iwasawa theory}

Another example, as an application of the previous section, which should also be very interesting to study, are modules over Iwasawa algebras. I want to point out from the beginning that I really am not especially competent in the fine details of Iwasawa theory. In fact I know very little, but the ingredients in the construction below are rather general so I think I got this part right at least. 

The idea here is simply an adaptation of the previous section when $A=\mathfrak{o}_K$ and $\Gam\simeq \mathbb{Z}_p$ acts on a $\mathfrak{o}_K$-module $M$. 

But let us briefly recall what the notations here are, just to be clear. Let $K$ be a finite unramified extension of $\mathbb{Q}_p$ (or $\mathbb{Q}$ for that matter) with rings of integers $\mathfrak{o}_K$ and consider the cyclotomic tower
$$K\subset K(\mu_p)\subset K(\mu_{p^2})\subset \cdots\subset K(\mu_{p^n})\subset \cdots\subset K(\mu_{p^\infty}).$$ As usual $\mu_{p^n}$ denotes the $p^n$th roots of unity in the algebraic closure $K^{\mathrm{alg}}$. The Galois group is isomorphic to $\mathbb{Z}_p^\times=(\mathbb{Z}/p)^\times\times \mathbb{Z}_p$, where $\mathbb{Z}_p$ is the additive group of $p$-adic integers. Put $K_\infty=K(\mu_{p^\infty})^{(\mathbb{Z}/p)^\times}$. Then $\Gam:=\Gal(K_\infty/K)$. 

So, let $M$ be a $\Gam$-representation over $\mathfrak{o}_K$. This means in particular that $M$ is a module over the Iwasawa algebra $\Lambda=\mathfrak{o}_K[\![\Gam]\!]$. It is well-known that $\Lambda$ is generated by a non-canonical topological generator $\gamma\in \Gam$. In addition there is an isomorphism $\Lambda\simeq \mathfrak{o}_K[\![T]\!]$ via $\id-\gamma\mapsto T$. Clearly, $T$ acts on any Iwasasa module $M$ as $\id-\gamma$ and hence for $f\in \Lambda$ we have representations
$$ f=\sum_{i=0}^\infty f_i T^i\quad \longleftrightarrow\quad f=\sum_{i=0}^\infty f_i(\id-\gamma).$$ Therefore we see that the first order terms $f_1$ are in fact $\gamma$-derivations! 

To have more flexibility we can generalize the situation a bit as follows. Let $B$ be an $\mathfrak{o}_K$-algebra endowed with a $\Gam$-action. Tensoring with $B$ we get a semi-linear action of $B[\![\id-\gamma]\!]=B\otimes_{\mathfrak{o}_K} \Lambda$ on $M_B:=B\otimes_{\mathfrak{o}_K} M$. Then every $b(\id-\gamma)$ is a $\gamma$-derivation on $M_B$.

Now, forming (for instance) $\mathrm{Sym}_B(M_B)$ and extending the $\Gam$ action in the obvious way, we can form the hom-Lie algebra structure
$$\underline{h}M_B(\gamma):=\mathrm{Sym}_B(M_B)\cdot \delta_{b_\gamma}, \quad \delta_{b_\gamma}:=b_\gamma(\id-\gamma), \quad b_\gamma\in B.$$
\emph{However}, I want to emphasize that this is not canonical since $\gamma$ is not really canonical. On the other hand the equivariant structure
$$\underline{h}\mathrm{Sym}(M_B)(\Gam):=\bigsqcup_{\gamma\in \Gam}\underline{h}M(\gamma):=\bigsqcup_{\gamma\in \Gam} \mathrm{Sym}(M_B)\cdot \delta_{b_\gamma}$$actually \emph{is} canonical (and huge!). 

The classification of Iwasawa modules is well-established. Namely, if $M$ is a $\Lambda$-module then 
$$M\overset{\mathrm{pseudo}}{\thicksim}\Lambda^r\oplus \bigoplus_{i=1}^s\frac{\Lambda}{p^{m_i}}\oplus\bigoplus_{j=1}^t\frac{\Lambda}{F_j^{n_j}},$$where $\overset{\mathrm{pseudo}}{\thicksim}$ means isomorphism up to finite kernels and cokernels (a so-called pseudo-isomorphism); the $F_j$ are unique irreducible polynomials (of a certain kind). The numbers $s,t, m_i, n_j$ are also unique. This means that the resulting (equivariant) hom-Lie algebras arising from Iwasawa theory could in principle be classified (up to finite kernels and cokernels of the underlying modules). 

Often one encounters set-ups where one is interested in a special element inside $K_1(R)$ for some ring $R$. The group $K_1(R)$ is generated by isomorphism classes of symbols $[P,\theta]$ where $P$ is a finitely generated projective $R$-module and $\theta$ is an automorphism on $P$. An isomorphism of symbols $[P,\theta_P]\to [Q,\theta_Q]$ is an isomorphism of modules $f: P\simeq Q$ such that $f\circ \theta_P=\theta_Q\circ f$. 

It is easy to see that every isomorphism class of symbols $[P,\theta]$ defines a \emph{canonical} (and up to multiples, unique) hom-Lie algebra structure over $\mathrm{Sym}(P)$, namely 
$$\underline{h}\mathrm{Sym}(P)(\theta):=\mathrm{Sym}(P)\cdot\delta_\theta,\quad \delta_\theta:=\id-\theta.$$

In the case of Iwasawa theory, the ring $R$ is some generalization of the Iwasawa algebra $\Lambda$. For instance, take a (finite or infinite) Galois extension $L$ of $K_\infty$ with (not necessarily commutative) Galois group  $\Gam^G$. Then $R$ often is $\mathfrak{o}_K[\![\Gam^G]\!]$, or the fraction field of this ring. Hence $K_1(\mathfrak{o}_K[\![\Gam^G]\!])$ is generated as a set by isomorphism classes of symbols $[P,\theta]$, where $P$ is a finitely generated and projective $\mathfrak{o}_K[\![\Gam^G]\!]$-module. 

\section{Hom-Lie algebras, $(\phi,\Gamma)$-modules and twisted derivations}
\subsection{Difference modules}\label{sec:differencemodules}
\begin{dfn}A \emph{$\sigma$-difference ring} is a ring $A$ together with a $\sigma\in\End(A)$; a \emph{$(\Phi,\sigma)$-difference module} $(M,\Phi,\sigma)$ is a module over a difference ring $(A,\sigma)$ together with a $\sigma$-linear endomorphism $\Phi$.
\end{dfn}
To recall, the notion $\sigma$-linear means that 
$$\Phi(am)=\sigma(a)\Phi(m),\quad \text{for}\,\, a\in A,\,\, m\in M.$$
\begin{prop}Every $\sigma$-difference operator $\sum_i a_i\sigma^i$ over a $\sigma$-difference ring $(A,\sigma)$ can uniquely be expressed as a $\sigma$-differential operator $\sum_i c_i\partial_\sigma$, where $\partial_\sigma:=\id-\sigma$, and vice versa.
\end{prop}The proposition follows immediately from the following (probably well-known) lemma.
\begin{lem}We have 
\begin{equation}\label{eq:combformula}\sigma^n=\sum_{i=0}^{n-1}\binom{n}{i}(\id-\sigma)^i+(-1)^n(\id-\sigma)^n.
\end{equation}
\end{lem}
\begin{proof}The proof is entirely an exercise in combinatorics, starting from the binomial identity $$(\id-\sigma)^n=\sum_{i=0}^n\binom{n}{i}\sigma^i,$$ and using induction on $n$. For details, see \cite{Larsson3}.
\end{proof}

Now, let $(\mathscr{E},\Phi)$ be a $\sigma$-difference module over a $\sigma$-difference ring $(\mathscr{A},\sigma)$ on a scheme $X$. Furthermore, for each $U\subseteq X$ we put $\partial_{\sigma;U}:=t_U(\id-\sigma)$ for $t\in \mathscr{A}(U)$. By abuse of notation we don't explicitly mention that $\Phi$ and $\sigma$ should really be $\Phi_U$ and $\sigma_U$. 

Consider the map
\begin{equation}\label{eq:connection}
\nabla^{(\sigma)} : \mathscr{E}\to \mathscr{A}\cdot\underline{\varepsilon}\otimes_{\mathscr{A}}\mathscr{E},\quad m\mapsto \underline{\varepsilon}\otimes t_U(\id-\Phi)(m),\quad t_U\in \mathscr{A}(U), 
\end{equation}for each open $U\subseteq X$ and where $A\cdot\underline{\varepsilon}$ is the canonical rank one $\mathscr{A}(U)$-module with basis $\underline{\varepsilon}$. This map satisfies a twisted Leibniz rule over every open $U\subseteq X$ (suppressing $U$ from the notation):
\begin{multline*}
\nabla^{(\sigma)}(am)=\underline{\varepsilon}\otimes\big( t(\id-\sigma)(a)m\big)+\underline{\varepsilon}\otimes\big(a^\sigma t(\id-\Phi)(m)\big)\\
=\big(t(\id-\sigma)(a)\cdot\underline{\varepsilon}\big)\otimes e+a^\sigma\cdot\underline{\varepsilon}\otimes t(\id-\Phi)(m),
\end{multline*}i.e.,
$$\nabla^{(\sigma)}(am)=\partial_\sigma(a)\cdot\underline{\varepsilon}\otimes m+a^\sigma\nabla^{(\sigma)}m.$$In addition, we can easily see that
$$\nabla^{(\sigma)}\circ \Phi=q\cdot\Phi\circ\nabla^{(\sigma)},$$
where $q=\sigma(t)/t$. 

For simplicity we consider the affine case from now on. 
\begin{lem}\label{lem:diffconnection}Keeping the notation from above, there is a canonical $\sigma$-twisted connection $\nabla^{(\sigma)}_M$ given by (\ref{eq:connection}). Conversely, localizing at $t$ if necessary, given a $\sigma$-twisted connection we have a canonical $\sigma$-difference module $(M_\nabla,\Phi)$ as the kernel of $\nabla^{(\sigma)}$. 
\end{lem}
\begin{remark}By twisting the actions of $A$ on $A\cdot\underline{\varepsilon}$ by the rule 
$$a\underline{\varepsilon}=\underline{\varepsilon}a^\sigma$$we can in fact transform the above twisted Leibniz rules into 'ordinary' Leibniz rules:
$$\nabla^{(\sigma)}(am)=\partial_\sigma(a)\cdot\underline{\varepsilon}\otimes m+a\nabla^{(\sigma)}m.$$This actually means that $A\cdot\underline{\varepsilon}$ actually is the first-order part of the twisted polynomial ring $A\{t,\sigma\}$, with $a\cdot \underline{\varepsilon}=\underline{\varepsilon}a^\sigma$. Notice that this is in fact the twisted polynomial ring $A[\theta;\sigma]$, with $a\theta=\theta\sigma(a)$.
\end{remark}
For $K$ a number field and $p$ a rational prime, we denote by $D_p$ the decomposition group of $p$ in $K$. 
\subsection{$p$-adic Hodge theory}
We need some basics from $p$-adic Hodge theory. The following recollection will be superficial at best. \\

\noindent\textbf{Notation.} From now on the following notation will be used.

Let $K_0$ be a finite unramified extension of $\mathbb{Q}_p$ and let $K$ be a finite totally ramified extension of $F$. In fact, if $k$ is the residue class field of $K$ then $K_0=W(k)[p^{-1}]$. For any field $F$, the designation $F^{\mathrm{unr}}$ means the maximal unramified extension of $F$ inside its algebraic closure. Throughout $V$ will be a $p$-adic representation of $G_K:=\Gal(K^\mathrm{alg}/K)$, i.e., a $\mathbb{Q}_p$-vector space with a continuous action of $G_K$.

For any (topological) $\mathbb{Q}_p$-algebra $B$, we say that $V$ is \emph{$B$-admissible} if $B\otimes_{\mathbb{Q}_p}V=B^d$ as $B[G_K]$-modules. Here $d:=\dim_{\mathbb{Q}_p}(V)$. 

There is a $\mathbb{Q}_p$-algebra $\mathbf{B}_\mathrm{cris}$ endowed with a continuous action of $G_K$ and a Frobenius morphism $\phi$. We won't define this here, merely refer to \cite{Berger0}, which is an excellent introduction to $p$-adic representation theory and $p$-adic Hodge theory.  There is an important invertible element inside $\mathbf{B}_\mathrm{cris}$, namely $t:=\log([\epsilon])$, where $[\epsilon]$ is the Teichm\"uller lift of a coherent sequence of $p^n$-th roots of unity in a certain canonical characteristic $p$ ring. We refer to \cite{Berger0} or \cite{Berger1} for the details. 

A $p$-adic representation is called \emph{crystalline} if it is $\mathbf{B}_\mathrm{cris}$-admissible. Associated to every crystalline representation is crystalline Dieudonn\'e module defined as 
$$D_\mathrm{cris}(V):=\big(\mathbf{B}_\mathrm{cris}\otimes_{\mathbb{Q}_p}V\big)^{G_K}.$$ This is a filtered $\phi$-module over $K_0$. 

Put $K_\infty:=\cup_{n\geq 0}K(\mu_{p^n})$, where $K(\mu_{p^n})$ denotes the field, generated over $K$, by the $p^n$-th roots of unity in $K^{\mathrm{alg}}$. We have the following map of extensions
$$\xymatrix{K^{\mathrm{alg}}\\
K_\infty\ar@{-}[u]\ar@/_1pc/@{..}[u]_{H_K}\\
K\ar@{-}[u]\ar@/_1pc/@{..}[u]_{{\Gamma}_K}\ar@/^1pc/@{..}[uu]^{G_K}\\
K_0\ar@{-}[u]}$$Hence $H_K=\Gal(K^{\mathrm{alg}}/K_\infty)$ and $\Gamma_K=G_K/H_K$.

We define the rings $\mathbf{A}^+_K$, $\mathbf{B}^+_K$,  $\mathbf{A}_K$ and $\mathbf{B}_K$ as 
$$\mathbf{A}_K^+:=\mathfrak{o}_K\big[[\pi]\big],\qquad \mathbf{B}_K^+:=\mathbf{A}_K^+[p^{-1}]$$
and 
$$\mathbf{A}_K:=\widehat{\mathfrak{o}_K\big[[\pi]\big][\pi^{-1}]}_{\hat{p}},\qquad \mathbf{B}_K:=\mathbf{A}_K[p^{-1}],$$where the subscript $\hat{p}$ means that the completion is taken with respect to $p$. There are canonical injections $\mathbf{A}^+_K\hookrightarrow \mathbf{A}^+$ and $\mathbf{B}_K^+\hookrightarrow \mathbf{B}^+$. Furthermore, let $\mathbf{B}$ be the $p$-adic completion of $\mathbf{B}_K^{\mathrm{unr}}$ inside a certain ring $\tilde{\mathbf{B}}$ that won't be defined here. 

For every $r>0$ there is a ring $\mathbf{B}^{\dagger,r}$ inside $\mathbf{B}$ (we won't define $\mathbf{B}^{\dagger,r}$ either), and we put
\begin{align*}
\mathbf{B}^{\dagger,r}_K&:=(\mathbf{B}^{\dagger,r})^{H_K}\\
&=\Big\{\sum_{-\infty}^{\infty} a_i\pi^i\,\,\Big\vert a_i\,\,\in F^{\mathrm{unr}},\,\, \sum_{-\infty}^{\infty} a_iT^i,\,\, \text{convergent and bounded on}\,\, A[p^{\frac{1}{e_K r}},1)\Big\}.
\end{align*}Here $e_K$ is the absolute ramification index of $K/K_0$ and $A[a,b)$ denotes all $x\in K$ such that $a\leq |x|<b$. Further, put $\mathbf{B}^\dagger_K:=\cup_{r>0}\mathbf{B}^{\dagger,r}_K$. All of the above rings, except $\mathbf{B}^{\dagger,r}_K$, support a Frobenius morphism, which we denote $\phi$ in all cases. 

A $\mathbf{B}_K$-vector space $M$ with a semilinear action of $\phi$ is called a \emph{$(\phi,\Gamma)$-module}; if the canonical $\phi^\ast M\to M$ is an isomorphism, $M$ is an \emph{\'etale} $(\phi,\Gamma)$-module. This last is equivalent to $M$ having a $\phi$-stable basis. 

A basic theorem of J.-M. Fontaine asserts that there is an equivalence of categories between the category of $G_K$-representations and \'etale $(\phi,\Gamma)$-modules, given by
$$	V \mapsto D(V):=(\mathbf{B}\otimes_{\mathbb{Q}_p} V)^{H_K},\quad\text{with inverse}\quad	D\mapsto (\mathbf{B}\otimes_{\mathbf{B}_K} V)^{\phi=\id}.
$$ Here, $(?)^{\phi=\id}$ means taking invariants under $\phi$. 

By replacing $\mathbf{B}$ and $\mathbf{B}_K$ in the above functors by the rings $\mathbf{B}^{\dagger,r}$, $\mathbf{B}^{\dagger,r}_K$, $\mathbf{B}^{\dagger}$ and $\mathbf{B}^\dagger_K$, we define similarly
$$	V \mapsto D^{\dagger,r}_K(V):=(\mathbf{B}^{\dagger,r}\otimes_{\mathbb{Q}_p} V)^{H_K},\quad\text{with inverse}\quad	D\mapsto (\mathbf{B}^{\dagger,r}\otimes_{\mathbf{B}^{\dagger,r}_K} V)^{\phi=\id}.
$$and
$$	V \mapsto D^{\dagger}_K(V):=(\mathbf{B}^{\dagger}\otimes_{\mathbb{Q}_p} V)^{H_K}=\bigcup_{r>0}D^{\dagger,r}_K(V),\quad\text{with inverse}\quad	D\mapsto (\mathbf{B}^{\dagger}\otimes_{\mathbf{B}^{\dagger}_K} V)^{\phi=\id}.
$$A result by F. Cherbonnier and P. Colmez \cite{CherbonnierColmez} gives that $D_K(V)$ can be obtained as 
$$D_K(V)=\mathbf{B}\otimes_{\mathbf{B}^{\dagger,r}_K}D^{\dagger,r}_K(V), \quad\text{for $r$ big enough,}$$and $V\mapsto D^{\dagger}_L(V)$ is an equivalence of categories between $G_K$-representations and \'etale $(\phi,\Gamma)$-modules over $\mathbf{B}^\dagger_K$. We need an extension ring of $\mathbf{B}^\dagger_K$, namely, $\mathbf{B}^\dagger_{\mathrm{rig},K}$. This is defined in the same way as $\mathbf{B}^\dagger_K$ but without the boundedness condition:
\begin{align*}
\mathbf{B}^{\dagger,r}_K:=\Big\{\sum_{-\infty}^{\infty} a_i\pi^i\,\,\Big\vert a_i\,\,\in K_0^{\mathrm{unr}},\,\, \sum_{-\infty}^{\infty} a_iT^i,\,\, \text{convergent on}\,\, A[p^{\frac{1}{e_K r}},1)\Big\}.
\end{align*}We put $\mathbf{B}^\dagger_{\mathrm{rig},K}:=\cup_{r>0}\mathbf{B}^{\dagger,r}_{\mathrm{rig},K}$ and 
$$V\mapsto D_{\mathrm{rig},K}^\dagger(V):=\mathbf{B}^\dagger_{\mathrm{rig},K}\otimes_{\mathbf{B}^\dagger_K}D^\dagger(V).$$
The ring $\mathbf{B}^\dagger_{\mathrm{rig},K}$ is a domain.

\subsection{Rigid analytic spaces}
For more details on what follows next, I recommend P. Berthelot's preprint \cite{BerthelotRigid} or J. Nicaise's introduction \cite{Nicaise}. 

Recall that a \emph{Tate algebra} over $K$ is 
$$T_m:=K\{t_1,t_2,\dots, t_m\}=\Big\{\sum_{\underline{i}\in\mathbb{N}^m}a_{\underline{i}}t^{\underline{i}}\quad\big\vert \quad a_{\underline{i}}\in K^m, \,\, |a_{\underline{i}}|\to 0, \,\,\text{as}\,\, |\underline{i}|\to \infty\Big\}.$$ A quotient of a Tate algebra is called an \emph{affinoid algebra}. Let $\mathrm{Spm}(A)$ denote the maximal spectrum (i.e., the set of maximal ideals of $A$) of an algebra $A$. An \emph{affinoid space} or \emph{affine rigid space} is a space of the form $\mathrm{Spm}(A)$ for some affinoid algebra $A$. Loosely speaking, a \emph{rigid analytic space} is a topological space, given by glueing affinoid spaces. 

Let $X$ be a scheme over the valuation ring $\mathfrak{o}_K$ of a local field $K$. Completing $X$ at the closed fibre yields a formal scheme $X_\infty$. We define the \emph{rigid space} associated to $X$ as the generic fibre of $X_\infty$ and denote it $X^{\mathrm{an}}$. 

We won't use rigid analytic spaces other than affinoid spaces, so for simplicity we call affinoid spaces simply as \emph{rigid spaces}. 
\subsection{Rigid spaces and families of $(\phi,\Gamma)$-modules}
This section more or less follows \cite{KedlayaLiu}, but see also \cite{Hellmann}. Unadorned completed tensor products are over $\mathbb{Q}_p$. 

Let $A(0,s]\subset\mathbb{C}_p$ be an annulus with $s$ sufficiently large. We consider $\mathbf{B}^{\dagger,s}_{\mathrm{rig},K}$ as the ring of functions on $A(0,s]$, and $\mathbf{B}^{\dagger}_{\mathrm{rig},K}$ as a ring of functions on $A(0,1)=\dirlim A(r,1)\subset\mathbb{C}_p$.  More specifically, we set
\begin{align*}
\mathscr{B}^{\dagger,s}_{K}&:=H^0\big(A(0,s],\mathscr{O}_{A(0,s]}\big)\\
\mathscr{B}^{\dagger}_{\mathrm{rig},K}&:=H^0\big(A(0,1),\mathscr{O}_{A(0,1)}\big).
\end{align*}Then $\mathbf{B}^{\dagger,s}_K$ and $\mathbf{B}^\dagger_{\mathrm{rig},K}$ are the algebras of global sections of $\mathscr{B}^{\dagger,s}_{K}$ and $\mathscr{B}^{\dagger}_{\mathrm{rig},K}$, respectively.

We will consider the product $X\times A$ of rigid spaces, where $A$ is any annulus. Notice that $$\mathrm{pr}_{X,\ast}\mathscr{O}_{X\times A}=\mathscr{O}_X\hat{\otimes} H^0(A,\mathscr{O}_A).$$
\begin{dfn}Let $X=\mathrm{Spm}(A)$, with $A$ an affinoid algebra, be a rigid space with structure sheaf $\mathscr{O}_X$. We define a \emph{family of $(\phi,\Gamma)$-modules} over $X$ to be a locally free sheaf $\mathscr{B}$ of $\mathscr{O}_X\hat{\otimes}\mathscr{B}^{\dagger}_{\mathrm{rig},K}$-modules on $X$, together with semilinear commuting actions of $\phi$ and $\Gamma$, such that $\phi^\ast\mathscr{B}\xrightarrow{\approx}\mathscr{B}$. If $\mathscr{B}$ is induced from a $\mathscr{O}_X\hat{\otimes}\mathscr{B}^{\dagger,r}_{K}$-module $\mathscr{B}^\circ$ by extension of scalars, $\mathscr{B}$ is said to be \emph{\'etale} and $\mathscr{B}^\circ$ is the \emph{model} for $\mathscr{B}$.
\end{dfn}

Corollary 6.6 of \cite{KedlayaLiu} states that an \'etale model is unique assuming it exists. The following result can be found in \cite{KedlayaLiu} (Theorem 3.11 combined with Definition 3.12):
\begin{prop}\label{prop:KedlayaLiuTheorem}Let $\mathscr{E}$ be a locally free $G_K$-representation over $X$. Then there is an $r'\geq 0$ such for all $r\geq r'$ there exists a module $D^{\dagger,r}_K(\mathscr{E})$ over $\mathscr{O}_X\hat{\otimes}\mathscr{B}^{\dagger,r}_{K}$ satisfying
\begin{itemize}
	\item[(a)] $\mathrm{rk}(D^{\dagger,r}_K(\mathscr{E}))=\mathrm{rk}(\mathscr{E})$;
	\item[(b)] for any $\mathfrak{m}\in X$, the natural map
	$$D^{\dagger,r}_K(\mathscr{E})\otimes_{\mathscr{O}_X}k(\mathfrak{m})\longrightarrow D^{\dagger,r}_K(\mathscr{E}_\mathfrak{m})$$ is an isomorphism.
\end{itemize}
\end{prop}

The definition of the module $D^{\dagger,r}_K(\mathscr{E})$ is rather complicated and can be found in \cite{BergerColmez}. We only need existence here. 

Define 
$$D^\dagger_K(\mathscr{E}):=D^{\dagger,r}_K(\mathscr{E})\underset{\mathscr{O}_X\hat{\otimes}
\mathscr{B}^{\dagger,r}_K}{\bigotimes}(\mathscr{O}_X\hat{\otimes} \mathscr{B}^{\dagger}_K),\quad \text{for $r$ big enough.}$$This is a family of \'etale $(\phi,\Gamma)$-modules over $\mathscr{O}_X\hat{\otimes} \mathscr{B}^{\dagger}_K$. Similarly we put
$$D^{\dagger}_{\mathrm{rig},K}(\mathscr{E}):=D^\dagger_K(\mathscr{E})\underset{\mathscr{O}_X\hat{\otimes}
\mathscr{B}^\dagger_K}{\bigotimes}(\mathscr{O}_X\hat{\otimes} \mathscr{B}^\dagger_{\mathrm{rig},K})$$ and this defines a family of \'etale $(\phi,\Gamma)$-modules over $\mathscr{O}_X\hat{\otimes} \mathscr{B}^\dagger_{\mathrm{rig},K}$.

\subsection{$p$-adic hom-Lie algebras and $(\phi,\Gamma)$-modules}
To add more structure, it is sometimes beneficial to endow the modules involved with an algebra structure. Obviously, such a structure is not in any way unique or canonical, so different algebra structures give rise to different hom-Lie algebra structures. 

If $\mathscr{E}$ is an $\mathscr{O}_{X^\mathrm{an}}$-algebra, then $D^\dagger_{\mathrm{rig}}(\mathscr{E})$ inherits this algebra structure. \emph{However}, the actions of $\phi$ and $\Gamma$ and their compatibility impose serious restrictions on the possible algebra structures on $D^\dagger_{\mathrm{rig}}(\mathscr{E})$. It might be (or is likely?) impossible to descend the given algebra structure on $\mathscr{E}$ to one on $D^\dagger_{\mathrm{rig}}(\mathscr{E})$, \emph{compatible} with the actions of $\phi$ and $\Gamma$.

Recall that by a $K$-rigid space we mean an affinoid variety over $K$, where $K$ is a local field of characteristic zero. A \emph{Frobenius} (\emph{morphism}) on $X^{\mathrm{an}}$ is a lift $\phi: \mathscr{O}_{X^\mathrm{an}}\to\mathscr{O}_{X^\mathrm{an}}$ of the Frobenius on $K$. 
\begin{dfn}Let $X^{\mathrm{an}}$ be a rigid space over $K$, $\phi$ a lift of Frobenius. We define a (very) \emph{fake Frobenius structure} on $X^{\mathrm{an}}$ to be a quadruple $(\mathscr{M}, \mathscr{A}, \phi, \nabla^{(\gamma)})$ where
\begin{itemize}
\item[(i)] $\mathscr{A}$ is an $\mathscr{O}_{X^\mathrm{an}}$-algebra with commuting actions of $\phi$ and $\gamma\in\Aut(\mathscr{A})$;
\item[(ii)] $\mathscr{M}$ an $\mathscr{A}$-module with commuting actions of $\phi$ and $\gamma\in\Aut(\mathscr{M})$, compatible with the actions on $\mathscr{A}$;
\item[(iii)]$\nabla^{(\gamma)}$ is an $\mathscr{O}_{X^\mathrm{an}}$-linear operator on $\mathscr{M}$ satisfying
$$\nabla^{(\gamma)}(a.m)=\nabla^{(\gamma)}_A(a).m+\gamma(a).\nabla^{(\gamma)}(m),$$ for some restriction $\nabla^{(\gamma)}_A$ of $\nabla^{(\gamma)}$ to $\mathscr{A}$;
\item[(iv)] $\phi$ and $\nabla^{(\gamma)}$ are compatible in the sense
$$\nabla^{(\gamma)}\circ \phi=p\cdot\phi\circ\nabla^{(\gamma)}.$$
\end{itemize}Clearly, this definition is independent on the assumption that $X^{\mathrm{an}}$ is a rigid space; we could equally well have assumed it to be an ordinary scheme.
\end{dfn}

The inclusion of the term ``Frobenius structure'' in the name comes from the fact that this structure visibly resembles a ``true'' Frobenius structure. We can ``re-twist'' the Leibniz rule in (iii) as in \ref{sec:differencemodules} by introducing a twisted polynomial algebra. However, in this case we have to accept something to the effect of ``non-commutative differential forms''.

We denote by $\chi$ the cyclotomic character $\chi: G_K\to \mathbb{Z}_p^\ast$. It is known, see \cite[p. 222]{Berger1} for instance, that 
$$\phi(t)=pt,\quad\text{and}\quad \gamma(t)=\chi(\gamma)t\quad \text{on}\quad \mathbf{B}^\dagger_{\mathrm{rig},K}.$$
\begin{thm}
Let $X^{\mathrm{an}}$ be a rigid space and let $\mathscr{E}$ be a torsion-free $\mathscr{O}_{X^{\mathrm{an}}}$-module. Assume that there is an $\mathscr{O}_{X^{\mathrm{an}}}$-linear action of $G_K:=\Gal(K^{\mathrm{alg}}/K)$ on $\mathscr{E}$. Fix $\sigma\in G_K$. 
\begin{itemize}
\item[(a)] There is a family $D^\dagger_{\mathrm{rig},K}(\mathscr{E})$ of $(\phi,\Gamma)$-modules over $\mathscr{O}_{X^\mathrm{an}}\hat\otimes \mathscr{B}^\dagger_{\mathrm{rig},K}[t^{-1}]$ together with two canonical global hom-Lie algebra structures  $$\SDer_\gamma(D^\dagger_{\mathrm{rig},K}(\mathscr{E}))\quad\text{and}\quad \SDer_\phi(D^\dagger_{\mathrm{rig},K}(\mathscr{E})),\quad\text{over}\quad \mathscr{O}_{X^\mathrm{an}}\hat{\otimes}\mathscr{B}^\dagger_{\mathrm{rig},K}[t^{-1}],$$ where $\gamma\in \im(G_K\twoheadrightarrow \Gamma)$, generated by 
$$\nabla^{(\gamma)}:=\frac{\id-\gamma}{(1-\chi(\gamma))t}\quad\text{and}\quad \nabla^{(\phi)}:=\frac{\id-\phi}{(1-p)t}.$$ Furthermore, these structures are independent of the choice of $\gamma$.
\item[(b)] These operators satisfy
$$\nabla^{(\phi)}\circ\phi=p\cdot \phi\circ\nabla^{(\phi)},\quad \nabla^{(\gamma)}\circ \gamma=\chi(\gamma)\cdot\gamma\circ\nabla^{(\gamma)},$$ 
$$\nabla^{(\gamma)}\circ\phi=p\cdot\phi\circ\nabla^{(\gamma)},\quad \nabla^{(\phi)}\circ \gamma=\chi(\gamma)\cdot\gamma\circ\nabla^{(\phi)},$$and
\begin{align*}
\nabla^{(\phi)}\circ\nabla^{(\gamma)}&=t^{-1}\nabla^{(\gamma)}-p^{-1}t^{-1}\nabla^{(\gamma)}
\circ\phi,\\ \nabla^{(\gamma)}\circ\nabla^{(\phi)}&=t^{-1}\nabla^{(\phi)}-\chi(\gamma)t^{-1}\nabla^{(\phi)}
\circ\gamma.\end{align*}
\item[(c)] The quadruple 
$$\big(D^\dagger_{\mathrm{rig},K}(\mathscr{E}_\mathfrak{m}), 
\mathscr{O}_{X^\mathrm{an}}\hat\otimes \mathscr{B}^\dagger_{\mathrm{rig},K}[t^{-1}],\phi,\nabla^{(\gamma)}\big)$$ defines a family of fake Frobenius structures on $X^\mathrm{an}$.
\end{itemize}
\end{thm}Notice that this means that there are actions of the global hom-Lie algebras $$\SDer_\phi(D^\dagger_{\mathrm{rig},K}(\mathscr{E}))\quad\text{and}\quad \SDer_\gamma(D^\dagger_{\mathrm{rig},K}(\mathscr{E}))$$ on the sheaf $D^\dagger_{\mathrm{rig},K}(\mathscr{E})$.

As we will see in the proof, I cheated a bit in the formulation of the theorem in that we have to invert $t$, i.e., we will consider
$$D^\dagger_{\mathrm{rig},K}(\mathscr{E}_\mathfrak{m})[t^{-1}]=D^\dagger_{\mathrm{rig},K}(\mathscr{E}_\mathfrak{m})\otimes_{\mathbf{B}^\dagger_{\mathrm{rig},K}}
\mathbf{B}^\dagger_{\mathrm{rig},K}[t^{-1}]$$ as a module over $\mathbf{B}^\dagger_{\mathrm{rig},K}[t^{-1}]$.
\begin{proof}By definition $D^\dagger_{\mathrm{rig},K}(\mathscr{E}_\mathfrak{m})$ is a $\mathbf{B}^\dagger_{\mathrm{rig},K}$-module for every $\mathfrak{m}\in X^{\mathrm{an}}$. This module is invariant under $H_K$ so any $\sigma\in G_K$ maps to a class $\sigma H_K\in \Gamma$ and the action of this class is independent of representative. Pick one $\gamma\in \sigma H_K$. Extend $\mathbf{B}^\dagger_{\mathrm{rig},K}$ by inverting $t$, i.e., $\mathbf{B}^\dagger_{\mathrm{rig},K}[t^{-1}]$. Then $(\id-\gamma)(t)=(1-\chi(\gamma))t$, which is invertible in $\mathbf{B}^\dagger_{\mathrm{rig},K}[t^{-1}]$. Hence $$\SDer_\gamma(D^\dagger_{\mathrm{rig},K}(\mathscr{E}_\mathfrak{m}))=\mathbf{B}^\dagger_{\mathrm{rig},K}[t^{-1}]\cdot \nabla^{(\gamma)},\quad\text{with}\quad \nabla^{(\gamma)}:=\frac{\id-\gamma}{(1-\chi(\gamma))t},$$by Lemma \ref{thm:classtwistder}. An easy calculation shows that 
$\nabla^{(\gamma)}\circ \gamma=\chi(\gamma)\cdot \gamma\circ\nabla^{(\gamma)}$. Indeed,
\begin{multline}\label{eq:ComNablaGamma}
\begin{split}
\gamma\circ\nabla^{(\gamma)}=((1-\chi(\gamma)))^{-1}\gamma\big(t^{-1}(\id-\gamma)\big)=
((1-\chi(\gamma)))^{-1}\gamma(t^{-1})\gamma(\id-\gamma)\\
=((1-\chi(\gamma)))^{-1}\frac{\gamma(t^{-1})}{t^{-1}}t^{-1}(\id-\gamma)\gamma=
\chi(\gamma)^{-1}\nabla^{(\gamma)}\circ \gamma.
\end{split}
\end{multline}

This defines a local hom-Lie algebra and globalizing the construction yields a global hom-Lie algebra. Notice that $\nabla^{(\gamma)}$ is up to a factor a first-order truncation of Sen's operator $\Theta$ as given in, for instance, \cite{Berger1}. 

For the other structure we take the Frobenius $\phi$ instead of $\gamma$. In this case we have $\phi(t)=pt$ and so $(\id-\phi)(t)=(1-p)t$, which is also invertible in $\mathbf{B}^\dagger_{\mathrm{rig},K}[t^{-1}]$. Thus the same construction applies and, as above, and we get
$$\SDer_\phi(D^\dagger_{\mathrm{rig},K}(\mathscr{E}_\mathfrak{m}))=\mathbf{B}^\dagger_{\mathrm{rig},K}[t^{-1}]\cdot \nabla^{(\phi)}, \qquad \nabla^{(\phi)}:=\frac{\id-\phi}{(1-p)t}$$ By the exact same arguments as in (\ref{eq:ComNablaGamma}) we prove $\nabla^{(\phi)}\circ \phi=p\cdot\phi\circ\nabla^{(\phi)}$ and 
$$\nabla^{(\gamma)}\circ \phi=p\cdot \phi\circ\nabla^{(\gamma)},\quad\text{and}\quad 
\nabla^{(\phi)}\circ \gamma=\chi(\gamma)\cdot \gamma\circ\nabla^{(\phi)}.$$This proves (a). Part (b) follows by the same type of arguments as in (\ref{eq:ComNablaGamma}), and (c) is clear from (a), (b) and from the discussion in section \ref{sec:differencemodules}.  
\end{proof}
\subsection{The generic example}In order to illustrate the above theorem we compute the generic example, in the sense that all others can be obtained from this by specializing the parameters. For simplicity we do the local case. The global case is obtained from the local by simply globalizing. 

Put $W:=\mathscr{E}_\mathfrak{m}$. We can write $D^\dagger_{\mathrm{rig},K}[t^{-1}](W)$ as 
$$D:=D^\dagger_{\mathrm{rig},K}[t^{-1}](W)=\bigoplus_{i=1}^{k}\mathbf{B}^\dagger_{\mathrm{rig},K}[t^{-1}]\underline{e}_i,$$where $k=\mathrm{rk}(W)$. Every element of $D$ can be written as linear combinations of elements on the form $a t^i \underline{e}_n$, with $a\in \mathbf{B}^\dagger_{\mathrm{rig},K}[t^{-1}]$, $i\in \mathbb{Z}$, and $t\nmid a$. 

 For simplicity of notation, we skip mentioning the factors $(1-p)$ and $(1-\chi(\gamma))$ in $\nabla^{(\phi)}$ and $\nabla^{(\gamma)}$, respectively.

Beginning with $\phi$, we compute
\begin{equation}\label{eq:genexphi}
\llangle at^i\underline{e}_n\cdot\nabla^{(\phi)},bt^j\underline{e}_m\cdot\nabla^{(\phi)}\rrangle\\
=t^{i+j}\big(p^i\phi_K(a)b-p^j\phi_K(b)a\big)\cdot\nabla^{(\phi)}.
\end{equation}

Now, as for the case of $\gamma$ we compute
\begin{equation}\label{eq:genexgamma}
\llangle at^i\underline{e}_n\cdot\nabla^{(\gamma)},bt^j\underline{e}_m\cdot\nabla^{(\gamma)}\rrangle\\
=t^{i+j}\big(\chi(\gamma)^i\phi_K(a)b-\chi(\gamma)^j\phi_K(b)a\big)\cdot\nabla^{(\gamma)}.\end{equation}
Notice the similarities between the two cases. 

\subsection{Crystalline representations}
When the fibre $\mathscr{E}_{\mathfrak{m}}$ is crystalline we can say more explicitly what the corresponding hom-Lie algebras look like. The reason for this is that associated to a crystalline representation is a $(\phi,\Gamma)$-module over a much simpler ring, namely $\mathbf{B}_K^+$. The resulting module is called a Wach module. 

Currently there is no theory of Wach modules over absolutely ramified fields $K/\mathbb{Q}_p$, so in this subsection we have to assume that $K=K_0=W(k)[p^{-1}]$. 

Let $V$ be a crystalline representation of $K$ with weights in $[a,b]\subset\mathbb{Z}$.
\begin{dfn}The \emph{Wach module} associated with $V$ is the unique finite-rank free $\mathbf{B}^+_K=\mathfrak{o}_K[p^{-1}\big[[\pi]\big][p^{-1}]=K\big[[\pi]\big][p^{-1}]$-module $\mathbf{N}(V)\subset D(V)=(\mathbf{B}\otimes_{\mathbb{Q}_p}V)^{H_K}$ of rank $\mathrm{rk}(\mathbf{N}(V))=\dim_{\mathbb{Q}_p}(V)$, satisfying
\begin{itemize}
\item[(a)] $\mathbf{N}(V)$ is stable under the action of $\Gamma$, and this action becomes trivial on $\mathbf{N}/\pi\mathbf{N}$;
\item[(b)] $\phi(\pi^{b}\mathbf{N})\subseteq \pi^{b}\mathbf{N}$, and
\item[(c)] $q^{b-a}$ annihilates $\pi^{b}\mathbf{N}/\langle\phi(\pi^{-a}\mathbf{N})\rangle$, where $q:=\phi(\pi)/\pi$ and $\langle\phi(\pi^{b}\mathbf{N})\rangle$ denotes the subgroup generated by $\phi(\pi^{-b}\mathbf{N})$.
\end{itemize}We filter the Wach module $\mathbf{N}(V)$ as
$$\mathrm{Fil}^i(\mathbf{N}):=\{x\in\mathbf{N}\mid \phi(x)\in q^i\mathbf{N}\}.$$
\end{dfn}
Notice that a Wach module is not necessarily stable under $\phi$. However, if $b=0$ (i.e., when $V$ \emph{positive} crystalline), $\mathbf{N}(V)$ becomes a $(\phi,\Gamma)$-module. This is in fact the case we will be mostly interested in.

The following theorem is proven in \cite{Berger3}.
\begin{thm}\label{thm:Berger3}
A $p$-adic $G_K$-representation $V$ with weights in $[a,b]$ is crystalline if and only if there is a Wach module $\mathbf{N}(V)\subset D(V)=(\mathbf{B}\otimes_{\mathbb{Q}_p}V)^{H_K}$. The associated functor $$\mathsf{Cris}(G_K)\to \mathsf{Wach}(\mathbf{B}^+_K),\quad V\mapsto \mathbf{N}(V)$$ is an equivalence of categories. Furthermore, there is an isomorphism of filtered $\phi$-modules
$$D_{\mathrm{cris}}(V)\xrightarrow{\approx}\mathbf{N}(V)/\pi\mathbf{N}(V)$$ when $V$ is crystalline. 
\end{thm}

Now, suppose that $V:=\mathscr{E}_\mathfrak{m}$ is crystalline with associated Wach module 
$$\mathbf{N}:=\bigoplus_{i=1}^kK\big[[\pi]\big]\underline{e}_i,$$ where $\{\underline{e}_i\}$ is a basis for $\mathbf{N}$ over $K\big[[\pi]\big]$. 

In view of the definition of a Wach module it is natural to enlarge $\mathbf{B}^+_K$ by inverting $\pi$, $$\mathbf{B}^+_{K,\pi}:=K\big[[\pi]\big][\pi^{-1}].$$

Over $\mathbf{B}^+_{K,\pi}$ we have $\nabla^{(\phi)}_{\pi}:=\pi^{-1}(\id-\phi)$ and $\nabla^{(\gamma)}_{\pi}:=\pi^{-1}(\id-\gamma)$. 

Remember that $V$ is assumed to be positive, meaning that all Hodge--Tate weights are in an interval $[a,0]\subset\mathbb{Z}$. 

Consider first, $\nabla^{(\phi)}_{\pi}:=\pi^{-1}(\id-\phi)$ and 
$$\SDer^\pi_\phi\big(K\big[[\pi]\big][\pi^{-1}]\big)=K\big[[\pi]\big][\pi^{-1}]\cdot \nabla^{(\phi)}_\pi\approx\bigoplus_{i\gg -\infty}K\pi^i \cdot\nabla^{(\phi)}_\pi$$acting on $K\big[[\pi]\big][\pi^{-1}]$. Notice that $\{\pi^i\cdot\nabla^{(\phi)}_\pi\}$ is a basis over $K$ for $\SDer^\pi_\phi\big(K\big[[\pi]\big][\pi^{-1}]\big)$. We extend the bracket in Theorem (\ref{thm:twistprod}) formally to $\SDer^\pi_\phi\big(K\big[[\pi]\big][\pi^{-1}]\big)$ with
$$\llangle a\pi^i\cdot\nabla^{(\phi)}_\pi, b\pi^j\cdot \nabla^{(\phi)}_\pi\rrangle=
\big(\phi(a\pi^i)\nabla^{(\phi)}_\pi(b\pi^j)-\phi(b\pi^j)\nabla^{(\phi)}_\pi(a\pi^i)\big)
\cdot\nabla^{(\phi)}_\pi.$$ Putting $\phi(\pi)=q\pi$ and calculating 
$$\nabla^{(\phi)}_\pi(\pi^i)=(1-q^i)\pi^{i-1},$$ this product can be computed to be
$$\llangle a\pi^i\cdot\nabla^{(\phi)}_\pi, b\pi^j\cdot \nabla^{(\phi)}_\pi\rrangle= 
\pi^{i+j}\big(q^i\phi_K(a)b-q^j\phi_K(b)a\big)\cdot\nabla^{(\phi)}_\pi.
$$Compare this with (\ref{eq:genexphi}).

The $\gamma$-case is simpler since $\gamma$ doesn't act on $K$: 
$$\SDer^\pi_\gamma\big(K\big[[\pi]\big][\pi^{-1}]\big)=K\big[[\pi]\big][\pi^{-1}]\cdot \nabla^{(\gamma)}_\pi\approx\bigoplus_{i\gg -\infty}K\pi^i \cdot\nabla^{(\gamma)}_\pi.$$We have, analogously as for $\phi$,  $\{\pi^i\cdot\nabla^{(\gamma)}_\pi\}$ as a basis over $K$ for $\SDer^\pi_\gamma\big(K\big[[\pi]\big][\pi^{-1}]\big)$, and we have 
$$\llangle\pi^i\cdot\nabla^{(\gamma)}_\pi, \pi^j\cdot \nabla^{(\gamma)}_\pi\rrangle=
\big(\gamma(\pi^i)\nabla^{(\gamma)}_\pi(\pi^j)-\gamma(\pi^j)\nabla^{(\gamma)}_\pi(\pi^i)\big)
\cdot\nabla^{(\gamma)}_\pi.$$Putting $\vartheta:=\gamma(\pi)$, we compute
$$\llangle \pi^i\cdot\nabla^{(\gamma)}_\pi, \pi^j\cdot\nabla^{(\gamma)}_\pi\rrangle=\big(\vartheta^i\pi^{j-1}-
\vartheta^j\pi^{i-1}\big)\cdot\nabla^{(\gamma)}_\pi.$$This time we compare with (\ref{eq:genexgamma}).

It is clearly interesting to study the effect of reduction with respect to $\pi$, i.e., what happens to the hom-Lie algebra structure on reduction
$$\mathbf{N}(V)\twoheadrightarrow \mathbf{N}(V)/\pi\mathbf{N}(V).$$By definition, we know that $\gamma$ acts trivially on $\overline{\mathbf{N}(V}):=\mathbf{N}(V)/\pi\mathbf{N}(V)$ so we loose one structure. However, the structure coming from the Frobenius $\phi$ is still highly interesting. 

By Berger's theorem \ref{thm:Berger3}, we know that $D_\mathrm{cris}(V)\approx \mathbf{N}(V)/\pi\mathbf{N}(V)$ as filtered $\phi$-modules over $K=K_0$. Reducing $\mathbf{N}(V)$ modulo $\pi$ gives $\overline{\mathbf{N}(V})$:
$$\mathbf{N}(V)=\bigoplus_{i=1}^kK\big[[\pi]\big]\underline{e}_i\longrightarrow
\overline{\mathbf{N}(V}):=\bigoplus_{i=1}^k K\underline{e}_i.$$In general, the action of the induced $\phi$ on $\overline{\mathbf{N}(V})$ is simpler than on $\mathbf{N}(V)$. The products of $\SDer^\pi_\phi(\mathbf{N}(V))$ reduces modulo $\pi$ to
$$
\llangle a \cdot\overline{\nabla}^{(\phi)}_\pi, 
b \cdot \overline{\nabla}^{(\phi)}_\pi\rrangle=\big(\phi(a)b-\phi(b)a\big)\cdot\overline{\nabla}^{(\phi)}_\pi
$$in the $K_0$-vector space $\SDer_\phi^\pi(\overline{\mathbf{N}^!(V}))^{\mathrm{red}}$.

\begin{example}In \cite{BergerLiZhu} is constructed an analytic (parametric) family of $2$-dimensional crystalline representations via the associated Wach modules. We will use this family to illustrate the above theory. 

First, fix $k\geq 2$ and put $m:=\lfloor (k-2)/(p-1)\rfloor$. 
The construction begins with two matrices $P(T), G(T)\in\mathrm{Mat}_2\big(\mathbb{Z}_p\big[[\pi,T]\big]\big)$ defined as
$$P(T):=\begin{pmatrix}
0 & -1\\
q^{k-1} & zT\end{pmatrix},\quad 
G(T):=\begin{pmatrix}
g^+ & 0\\
0 & g^-\end{pmatrix},$$ where $z:=z_0+z_1\pi+\cdots+z_{k-2}\pi^{k-2}\in\mathbb{Z}_p\big[[\pi]\big]$, $T$ is an  indeterminate, 
$$
\lambda^+:=\frac{\underset{n\geq 0}{\prod}\frac{\phi^{2n+1}(q)}{p}}{\gamma\Big(\underset{n\geq 0}{\prod}\frac{\phi^{2n+1}(q)}{p}\Big)}\quad\text{and}\quad \lambda^-:=\frac{\underset{n\geq 0}{\prod}\frac{\phi^{2n}(q)}{p}}{\gamma\Big(\underset{n\geq 0}{\prod}\frac{\phi^{2n}(q)}{p}\Big)},
$$and finally,
$$g^+:=\Big(\frac{\lambda^+}{\gamma(\lambda^+)}\Big)^{k-1},\quad g^-:=\Big(\frac{\lambda^-}{\gamma(\lambda^-)}\Big)^{k-1}.$$ Notice that $\lambda^+$ and $\lambda^-$ are power series in $\pi$. The point here is that $P(T)$ will represent the action of $\phi$ and $G(T)$ the action of $\gamma$. Since $\phi$ and $\gamma$ must commute, $P(T)$ and $G(T)$ must satisfy $P(T)\phi(G(T))=G(T)\gamma(P(T))$. In addition if $G^{\gamma'}(T)$ represents $\gamma'\in\Gamma$, then associativity of the action implies that we must impose $G^{\gamma\gamma'}(T)=G(T)\gamma(G^{\gamma'}(T))$. This clearly puts severe restrictions of the possible choices of $P(T)$ and $G(T)$. However, in \cite{BergerLiZhu} the authors prove that given $P(T)$ on the above form, there is a \emph{unique} $G(T)$ such that the above conditions are satisfied. The functions $g^+$ and $g^-$ are rather complicated so writing down an explicit expression seems (to me) to be very hard in general (even in specific cases). 

Let $\mathbf{N}(k,T)$ be a rank-two $\mathbb{Z}_p\big[[\pi]\big]$ with standard basis $\{\underline{e}_1,\underline{e}_2\}$. Then Berger et al show that the matrices $P(\alpha)$ and $G(\alpha)$ defines a Wach module structure on $\mathbf{N}(k,\alpha)$, for $\alpha\in\mathfrak{m}_{\mathbb{Q}_p}$ coming from a representation of Hodge--Tate weights $0$ and $1-k$. In explicit terms, 
$$\phi(\underline{e}_1)=q^{k-1}\underline{e}_2,\quad \phi(\underline{e}_2)=-\underline{e}_1+\alpha z\underline{e}_2,\quad \gamma(\underline{e}_1)=g^+\underline{e}_1,\quad\text{and}\quad \gamma(\underline{e}_2)=g^-\underline{e}_2.$$Remember that $q=\phi(\pi)/\pi$. The reduction $\overline{\mathbf{N}(k,\alpha})$ becomes
$$\phi(\underline{e}_1)=p^{k-1}\underline{e}_2,\quad\text{and}\quad \phi(\underline{e}_2)=-\underline{e}_1+\alpha p^m\underline{e}_2,$$as is easily seen. The filtration of $\overline{\mathbf{N}(k,\alpha})$ becomes 
$$\mathrm{Fil}^i(\overline{\mathbf{N}(k,\alpha}))=
\begin{cases}
	\overline{\mathbf{N}(k,\alpha}), & \text{if}\,\,\, i\leq 0\\
	\mathbb{Q}_p \underline{e}_1, & \text{if}\,\,\, 1\leq i\leq k-1\\
	0 & \text{if}\,\,\, i\geq k.
\end{cases}$$

 Put $a_p:=p^m\alpha$. One can prove \cite{Breuil} that $\overline{\mathbf{N}(k,\alpha})\approx D_{\mathrm{cris}}(V^\ast_{k,a_p})$, where $V^\ast_{k,a_p}:=\Hom_{\mathbb{Q}_p}(V_{k,a_p},\mathbb{Q}_p)$ for some irreducible, crystalline representation $V_{k,a_p}$ of Hodge--Tate weights $0$ and $k-1$. The representation $V_{k,a_p}$, and hence $\overline{\mathbf{N}(k,\alpha})$, can be viewed as a ``deformation'' of $V_{k,0}$, parametrized by multiples of the trace of Frobenius, i.e., $\mathrm{tr}(P(\alpha))$. 

Ok, now we compute. For $\SDer^\pi_\phi(\mathbf{N}(k,\alpha))$ we get:
\begin{align*}
&\llangle a\pi^i\cdot \nabla^{(\phi)}_\pi, b\pi^j\cdot \nabla^{(\phi)}_\pi\rrangle
=\pi^{i+j-1}\big(q^{i}\phi_K(a)b-q^j\phi(b)_Ka\big)\cdot\nabla^{(\phi)}_\pi;
\end{align*}for $\SDer^\pi_\gamma(\mathbf{N}(k,\alpha))$:
\begin{align*}
&\llangle \pi^i\cdot \nabla^{(\gamma)}_\pi, \pi^j\cdot \nabla^{(\gamma)}_\pi\rrangle
=\pi^{i+j-1}\big(\chi(\gamma)^i\phi_K(a)b-\chi(\gamma)^j\phi_K(b)a\big)\cdot\nabla^{(\gamma)}.
\end{align*}
The reduction $\SDer_\phi^\pi(\overline{\mathbf{N}(k,\alpha}))^{\mathrm{red}}$ simply becomes
\begin{align*}
&\llangle a\cdot \overline{\nabla}^{(\phi)}_\pi, b
\cdot \overline{\nabla}^{(\phi)}_\pi\rrangle=\big(\phi_K(a)b-\phi_K(b)a\big)\cdot\nabla^{(\phi)}_\pi.
\end{align*}

These hom-Lie algebras have the following actions on $\mathbf{N}(k,\alpha)$ and $\overline{\mathbf{N}(k,\alpha})$:
\begin{align*}&\SDer^\pi_\phi(\mathbf{N}(k,\alpha)): \qquad
\begin{cases}
a\pi^i\cdot\nabla^{(\phi)}_\pi(\underline{e}_1)=a\pi^{i-1}(1-q^{k-1})\underline{e}_2,\\ a\pi^i\cdot\nabla^{(\phi)}_\pi(\underline{e}_2)=a\pi^{i-1}\underline{e}_1+a\pi^{i-1}(1-\alpha z)\underline{e}_2;
\end{cases}\\
&\SDer^\pi_\gamma(\mathbf{N}(k,\alpha)): \qquad
\begin{cases} a\pi^i\cdot\nabla^{(\gamma)}_\pi(\underline{e}_1)=a\pi^{i-1}(1-g^+)\underline{e}_1,\\ a\pi^i\cdot\nabla^{(\gamma)}_\pi(e_2)=a\pi^{i-1}(1-g^-)\underline{e}_2;
\end{cases}\\
&\SDer_\phi^\pi(\overline{\mathbf{N}(k,\alpha}))^{\mathrm{red}}:\quad
\begin{cases}
a\overline{\nabla}^{(\phi)}_\pi(\underline{e}_1)=q(1-p^{k-1})\underline{e}_2,\\
a\overline{\nabla}^{(\phi)}_\pi(\underline{e}_2)=a\underline{e}_1+a(1-a_p)\underline{e}_2.
\end{cases}
\end{align*}
These all define parametric families of actions of hom-Lie algebras $\SDer_\phi^\pi(\mathbf{N}(k,\alpha))$ $\SDer_\gamma^\pi(\mathbf{N}(k,\alpha))$ and $\SDer_\phi^\pi(\overline{\mathbf{N}(k,\alpha}))^{\mathrm{red}}$, parametrized by $\alpha$, and in the case of the reduction, by the traces of Frobenius.
\end{example}
\section{Conclusions}
A popular theme in recent times (Bloch--Kato, Fontaine, Perrin-Riou, Colmez,...) have been to combine $p$-adic Hodge theory and Iwasawa theory (and this was actually why I started to think about all these possible applications of hom-Lie algebras). This is most easily seen by the extensive use of Fontaine's $(\phi,\Gam)$-modules in studying Iwasawa theory for crystalline and more generally de Rham representations of local Galois theory. 

Kato's philosophy that $L$-functions and $p$-adic representations are simply different points of observation of the same ``galaxy" provides some indication (in my opinion) that hom-Lie algebras might fit in this picture as a part of his galaxy, which I have tried to convey with the above simple and incomplete examples. I have a hope that the above can be of use and interest in this area. But who knows? 

\bibliographystyle{alpha}
\bibliography{refarithom}
\end{document}